\documentclass[a4paper, 11pt]{amsart} 
   
\usepackage[headings]{fullpage}               
\usepackage{boxedminipage}
\usepackage{amsfonts}
\usepackage{amsmath} 
\usepackage{amssymb}
\usepackage{graphicx}
\usepackage{amsthm}
\usepackage{subfig}
\usepackage{tikz}
\usetikzlibrary{cd}
\usepackage{setspace}
\usepackage[colorlinks=true,linkcolor=blue,urlcolor=blue,citecolor=red]{hyperref}
\usepackage[nameinlink]{cleveref}
\usepackage{enumerate}
\hypersetup{linktocpage=true,}
\newtheorem{thm}{Theorem}[section]
\newtheorem{lemma}[thm]{Lemma}
\newtheorem{conjecture}[thm]{Conjecture}
\newtheorem{proposition}[thm]{Proposition}
\newtheorem{corollary}[thm]{Corollary}

\theoremstyle{definition}
\theoremstyle{definition}
\theoremstyle{definition}

\newcommand{\rank}{\operatorname{rank}}

\newcommand{\val}{\operatorname{val}}

  \newcommand{\U}{{\mathcal{U}}}

\definecolor{colR}{rgb}{.932,.172,.172}
\definecolor{colB}{rgb}{.255,.41,.884}
\definecolor{colG}{rgb}{0,0.7,0}

\tikzstyle{vertex}=[circle, draw, fill=black, inner sep=0pt, minimum size=4pt]
\tikzstyle{smallvertex}=[circle, line width=1.5pt, draw, fill=black, inner sep=0pt, minimum size=2pt]

\tikzstyle{edge}=[line width=1pt]
\tikzstyle{dedge}=[edge,dashed,gray]
\tikzstyle{redge}=[edge,colR]
\tikzstyle{bedge}=[edge,colB]
\tikzstyle{gedge}=[edge,colG]
\tikzstyle{lnode}=[circle,white,draw, fill=black,inner sep=1pt, font=\scriptsize]
\tikzstyle{hollow}=[circle,gray,draw, thick, fill=white,inner sep=0pt, minimum size=4pt]

\let\emph\relax
\DeclareTextFontCommand{\emph}{\bfseries\em}

\begin{document}

\title{Coincident-point rigidity in normed planes}

\author[Sean Dewar]{Sean Dewar}
\address{Johann Radon Institute\\ Altenberger Strasse 69\\ 4040\\ Linz\\
Austria}
\email{sean.dewar@ricam.oeaw.ac.at}
\author[John Hewetson]{John Hewetson}
\address{Dept.\ Math.\ Stats.\\ Lancaster University\\
Lancaster LA1 4YF \\U.K. }
\email{j.hewetson2@lancaster.ac.uk}
\author[Anthony Nixon]{Anthony Nixon}
\address{Dept.\ Math.\ Stats.\\ Lancaster University\\
Lancaster LA1 4YF \\U.K.}
\email{a.nixon@lancaster.ac.uk}

\begin{abstract}
    A bar-joint framework $(G,p)$ is the combination of a graph $G$ and a map $p$ assigning positions, in some space, to the vertices of $G$. The framework is rigid if every edge-length-preserving continuous motion of the vertices arises from an isometry of the space. We will analyse rigidity when the space is a (non-Euclidean) normed plane and two designated vertices are mapped to the same position. This non-genericity assumption leads us to a count matroid first introduced by Jackson, Kaszanitsky and the third author. We show that independence in this matroid is equivalent to independence as a suitably regular bar-joint framework in a normed plane with two coincident points; this characterises when a regular normed plane coincident-point framework is rigid and allows us to deduce a delete-contract characterisation. We then apply this result to show that an important construction operation (generalised vertex splitting) preserves the stronger property of global rigidity in normed planes and use this to construct rich families of globally rigid graphs when the normed plane is analytic.
\end{abstract}

\maketitle

\section{Introduction}

A bar-joint \emph{framework} $(G,p)$ is the combination of a graph $G=(V,E)$ and a map $p:V\rightarrow \mathbb{R}^d$ assigning positions to the vertices of $G$ (and hence lengths to the edges). Intuitively, the framework is \emph{rigid} if every edge-length-preserving continuous motion of the vertices arises from an isometry of $\mathbb{R}^d$. More strongly, $(G,p)$ is \emph{globally rigid} if every framework in $\mathbb{R}^d$, on the same graph, with the same edge lengths actually has the same distance between every pair of vertices.

The rigidity and global rigidity of bar-joint frameworks in Euclidean spaces has been intensely studied in recent years (e.g. \cite{asi-rot,C2005,GHT,J&J,laman,NSW}) and has a rich history going as far back as classical work of Euler and Cauchy on Euclidean polyhedra. 
In the last decade, work on rigidity has been generalised to various non-Euclidean normed spaces (e.g. \cite{D19,DKN,DN,KL,kit-pow-1}). All of these results concern characterising the combinatorial nature of the `generic' behaviour of frameworks. This article extends this to frameworks with two points lying in the same location. The difficulty that already arises in this context shows how necessary the genericity assumption in those papers really was. 
Frameworks with coincident points have been considered in the Euclidean context \cite{FJK,Gul} and applied to global rigidity there \cite{CJ}, as well as for frameworks on surfaces \cite{JKN}.

Beyond the natural extension towards non-generic frameworks (and thus nearer to being of potential use in applications), we are motivated by the study of global rigidity in normed planes. The first and third author recently instigated research in this direction \cite{DN} proving global rigidity for an infinite class of graphs in analytic normed planes. In this paper we use our analysis of frameworks with two coincident points to improve
this result by creating a substantially richer class of globally rigid graphs.

We conclude the introduction with a short outline of what follows. After introducing the necessary background on the theory of rigid frameworks in normed planes, coincident point frameworks and the relevance notion of graph sparsity, in \Cref{sec:introdefs}, the majority of the paper is contained in \Cref{sec:recursive}. Here we provide a detailed geometric analysis of the effect of certain graph operations on the rigidity of a coincident point framework in a normed plane. In \Cref{sec:char} we combine these geometric results with combinatorial results of \cite{JKN} to establish a purely combinatorial characterisation of independence in the `coincident point normed plane rigidity matroid' and we deduce from this a delete-contract characterisation of coincident point rigidity in any strictly convex normed plane. In \Cref{sec:glob} we provide our other main results; these concern global rigidity. We deduce from our delete-contract characterisation that another graph operation preserves global rigidity, and we use this result alongside the results of \cite{DN} to establish global rigidity in the special case of analytic normed planes for a rich family of graphs.

\section{Rigidity and \texorpdfstring{$uv$}{uv}-coincident frameworks in normed spaces}
\label{sec:introdefs}

\subsection{Rigidity in normed spaces}

Let $X$ be a (real finite-dimensional) normed space with norm $\|\cdot\|$.
Unless stated otherwise, {\em we shall assume all normed spaces are not isometrically isomorphic to any Euclidean space}.
We define a \emph{support functional} of $z \in X$ to be a linear functional $f:X \rightarrow \mathbb{R}$ such that $f(z) = \|z\|^2$ and $\sup_{\|x\|=1} f(x) = \|z\|$.
It follows from the Hahn-Banach theorem that every point has a support functional and every linear functional of $X$ is the support functional of a point in $X$.
A non-zero point in $X$ is said to be \emph{smooth} if it has exactly one support functional,
and we shall denote the unique support functional of a smooth point $z$ by $\varphi_z$.
We say $X$ is \emph{smooth} if every non-zero point in $X$ is smooth,
and \emph{strictly convex} if every linear functional of $X$ is the support functional of at least one, and hence exactly one, point in $X$.
We note that for normed planes (2-dimensional normed spaces),
strict convexity is equivalent to the property that any two linearly independent smooth points have linearly independent support functionals.

Now let $(G,p)$ be a framework in $X$; that is the combination of a graph $G=(V,E)$ and a map $p:V\rightarrow X$ (called a \emph{placement} of $G$).
A \emph{finite flex} of $(G,p)$ is a continuous path $\alpha:[0,1] \rightarrow X^V$ where $\alpha(0) = p$ and $\|\alpha_x(t) - \alpha_y(t)\|= \|p_x-p_y\|$ for each edge $xy \in E$ and every $t \in [0,1]$.
If every framework $(G,\alpha(t))$ is congruent to $(G,p)$,
i.e.~there exists an isometry $f_t:X \rightarrow X$ so that $\alpha_x(t) = f_t(p_x)$ for every $x \in V$,
then we say $\alpha$ is \emph{trivial}.
We now define $(G,p)$ to be \emph{(continuously) rigid} if every finite flex of $(G,p)$ is trivial.

Since determining whether a framework is rigid is computationally challenging \cite{Abb},
we follow the literature and linearise the problem.
First, let $(G,p)$ be a \emph{well-positioned} framework,
i.e.~the point $p_x-p_y$ is smooth for each edge $xy\in E$.
An \emph{infinitesimal flex} of $(G,p)$ is a map $u:V\rightarrow X$ where $\varphi_{p_x-p_y}(u_x-u_y) =0$ for all $xy \in E$.
An infinitesimal flex is \emph{trivial} if there exists a linear map $T:X \rightarrow X$ and a point $z_0 \in X$ so that $u_x = T(p_x) +z_0$ for every vertex $x \in V$, and for every point $z \in X$ with support functional $f$ we have $f \circ T(z) =0$.
We now say that a well-positioned framework $(G,p)$ is \emph{infinitesimally rigid} if every infinitesimal flex of $(G,p)$ is trivial.

For a $d$-dimensional normed space $X$, a well-positioned framework $(G,p)$ in $X$, and a fixed basis $b_1,\ldots,b_d$ of $X$,
we can define the \emph{rigidity matrix} to be the $|E| \times d|V|$ matrix $R(G,p)$,
where for every $e \in E$, $x \in V$ and $i \in \{1,\ldots,d\}$ we have
\begin{align*}
    R(G,p)_{e,(x,i)} =
    \begin{cases}
        \varphi_{p_x-p_y}(b_i) &\text{if } e = xy,\\
        0 &\text{otherwise.}
    \end{cases}
\end{align*}
The choice of basis used to define $R(G,p)$ can be arbitrary as we are only interested in the sets of linearly independent rows of the matrix.
We say a well-positioned framework is \emph{independent} if $\rank R(G,p) =|E|$,
\emph{minimally (infinitesimally) rigid} if it is both independent and infinitesimally rigid,
and \emph{regular} if $\rank R(G,p) \geq \rank (G,q)$ for all other well-positioned frameworks $(G,q)$.
It is immediate that all independent and/or infinitesimally rigid frameworks are regular.
Given $k$ is the dimension of the linear space of trivial infinitesimal flexes of $(G,p)$,
it can be shown that so long as the affine span of the set $\{p_x :x \in V\}$ is $X$, the framework $(G,p)$ will be infinitesimally rigid if and only if $\rank R(G,p) = d|V| -k$; see \cite[Proposition 3.13]{D19}.
Consequently any well-positioned framework where the affine span of the set $\{p_x :x \in V\}$ is $X$, is  minimally rigid if and only if $|E| = \rank R(G, p) = d|V|-k$.

We can link infinitesimal rigidity to rigidity with the following result.
\begin{thm}\label{t:ar}
    Let $(G,p)$ be a well-positioned framework in a normed space $X$.
    \begin{enumerate}[(i)]
        \item \cite[Theorem 3.9]{dew21} If $(G,p)$ is infinitesimally rigid,
        then it is rigid.
        \item \cite[Theorem 1.1 \& Lemma 4.4]{D19} If $(G,p)$ is regular and rigid, and the set of smooth points in $X$ is open,
        then $(G,p)$ is infinitesimally rigid.
    \end{enumerate}
\end{thm}

We shall make use of the following perturbation result throughout the paper. It will be convenient to refer to properties of placements rather than frameworks. To this end we say that a placement $p$ of $G$ has property $P$ if the framework $(G,p)$ has property $P$.

\begin{lemma}[{\cite[Lemmas 4.1 and 4.4]{D19}}]\label{lem:perturb}
    For any graph $G$ and any normed space $X$,
    the set of well-positioned placements of $G$ in $X$ is a conull (i.e.~the complement of a set with Lebesgue measure zero) subset of $X^V$,
    and the set of regular placements of $G$ in $X$ is a non-empty open subset of the set of well-positioned placements.
\end{lemma}

We say that a graph is \emph{rigid} (respectively, \emph{independent}, \emph{minimally rigid}) if it has an infinitesimally rigid (respectively, independent, minimally rigid) placement.

Whether a graph $G = (V, E)$ is rigid/independent in a normed plane can be determined by simple sparsity counting conditions.
For $\emptyset \neq U \subseteq V$, $i_G(U)$ will denote the number of edges in the subgraph, $G[U]$, of $G$ induced by $U$.
For non-negative integers $k,\ell$, 
we say $G$ is \emph{$(k,\ell)$-sparse} if $i_G(U) \leq k|U| - \ell$ for every $\emptyset \neq U \subseteq V$ with $|U| \geq k$;
if $G$ is $(k,\ell)$-sparse and $|E|=k|V|-\ell$, then we say $G$ is \emph{$(k,\ell)$-tight}. 

Note that technically $(k,\ell)$-sparse graphs may have parallel edges and loops. However it is clear that independent graphs cannot, so we will assume throughout that all $(k,\ell)$-sparse graphs are simple.

\begin{thm}[\cite{dew2}]\label{thm:rigidne}
    A graph $G$ is minimally rigid in a normed plane $X$ if and only if $G$ is $(2,2)$-tight.
\end{thm}

For a family
\(\mathcal{S}=\{S_{1},S_{2},\dots ,S_{k}\}\) of subsets \(S_{i}\subseteq
V\), $1\leq i \leq k$, 
we say that ${\mathcal{S}}$ is a \emph{cover} of $F\subseteq E$ if $F\subseteq \{ xy :  \{x,y\}\subseteq S_i\
\hbox{for some}\ 1\leq i\leq k \}$. We can combine \Cref{thm:rigidne} with \cite[Section 3.1]{JKN} to obtain the following result.

\begin{corollary}\label{c:rankformula}
    Let $(G,p)$ be a well-positioned framework in a normed plane $X$.
    Let $\mathcal{S}$ be the set of all covers $\mathcal{X} := \{X_1,\ldots,X_k\}$.
    Given $s: \mathbb{N} \rightarrow \{0,1\}$ is the map with $s(x) = 1$ if $x=2$ and $s(x)=0$ otherwise,
    we have
    \begin{align*}
        \rank R(G,p) \leq \min_{\mathcal{X} \in \mathcal{S}} \sum_{i=1}^k \left( 2|X_i| - 2 - s(|X_i|) \right),
    \end{align*}
    with equality if and only if $(G,p)$ is regular. Moreover it suffices to minimise over all  covers $\mathcal{Y} := \{Y_1,\ldots,Y_k\}$ of the edge set $E$ where $|Y_i| \geq 2$ for each $i$ and $|Y_i \cap Y_j| \leq 1$ for all $i \neq j$,
    with equality only if $\min \{|Y_i|,|Y_j|\} = 2$.
\end{corollary}

\subsection{\texorpdfstring{$uv$}{uv}-coincident rigidity and \texorpdfstring{$uv$}{uv}-sparse graphs}

Let $G=(V,E)$  be a graph with vertices $u,v\in V$, and let $X$ be a normed space.
A framework $(G,p)$ in $X$ is \emph{$uv$-coincident} if $p_u=p_v$;
if the framework $(G-uv,p)$ is well-positioned, then we say that $(G,p)$ is a \emph{well-positioned} $uv$-coincident framework. Since $p_u=p_v$, we consider $G-uv$ so as to maintain smoothness of the support functionals associated with the framework; otherwise, no $uv$-coincident framework with $uv$ as an edge would be well-positioned.

A well-positioned $uv$-coincident framework $(G,p)$ is \emph{infinitesimally rigid} if $(G-uv,p)$ is infinitesimally rigid in $X$.
Given the linear space $X^V/uv := \{ q \in X^V : q_u = q_v\}$,
we say that a well-positioned $uv$-coincident framework $(G,p)$ is \emph{regular} if $\rank R(G-uv,p) \geq \rank R(G-uv,q)$ for all $q \in X^V/uv$,
and \emph{independent} if $uv \notin E$ and $(G,p)$ is independent in $X$.
A well-positioned $uv$-coincident framework $(G,p)$ is \emph{minimally (infinitesimally) rigid} if it is both infinitesimally rigid and independent.
We say a graph $G$ is \emph{$uv$-rigid} (respectively, \emph{$uv$-independent}, \emph{minimally $uv$-rigid}) if there exists a $uv$-coincident framework $(G,p)$ that is infinitesimally rigid (respectively, independent, minimally rigid).

By applying the same methods used to prove \Cref{lem:perturb},
we can obtain the natural analogue for $uv$-coincident frameworks.

\begin{lemma}\label{lem:uv-perturb}
    For any graph $G$ and any normed space $X$,
    the set of well-positioned $uv$-coincident placements of $G$ in $X$ is a conull (i.e.~the complement of a set with Lebesgue measure zero) subset of $X^V/uv$,
    and the set of regular $uv$-coincident placements of $G$ in $X$ is a non-empty open subset of the set of well-positioned $uv$-coincident placements.
\end{lemma}

As will be shown in \Cref{sec:char},
$uv$-rigidity in normed planes is closely related to the following sparsity property of graphs.

Let \(G=(V,E)\) be a graph and let $u,v$ be  two distinct vertices
of $G$.  Let $\mathcal{X}=\{X_1,X_2,...,X_k\}$ be a family with
$X_i\subseteq V$, $1\leq i\leq k$. We say that $\mathcal{X}$ is
\emph{\(uv\)-compatible} if \(u,v\in X_{i}\) and \(|X_{i}|\geq 3\)
hold for all $1\leq i\leq k$.  We define the \emph{value} of non-empty subsets
of $V$ and of \(uv\)-compatible families, denoted $\val(\cdot)$, as follows. For $\emptyset \neq U\subseteq V$, 
we let
$$\val(U)=2|U|-t_U,$$ where $t_U=4$ if $U= \{u,v\}$, $t_U=3$ if $U\neq \{u,v\}$ and $|U| \in \{2,3\}$, and $t_U=2$ otherwise.
For a
\(uv\)-compatible family \(\mathcal{X}=\{X_{1},X_{2},\dots
,X_{k}\}\) we let
$$\val(\mathcal{X})=\left(\sum_{i=1}^{k}\val(X_{i})\right)-2(k-1)=2 + \sum_{i=1}^{k}(2|X_i|-t_{X_i}-2).$$ Note that if
\(\mathcal{X}=\{U\}\) is a \(uv\)-compatible family containing only
one set then the two definitions agree, i.e.
$\val(\mathcal{X})=\val(U)$ holds.

We  say that $G$ is \emph{$uv$-sparse} if for all $U\subseteq V$ with
$|U|\geq 2$ we have $i_{G}(U)\leq \val(U)$ and for all
\(uv\)-compatible families $\mathcal{X}$ we have
$i_{G}(\mathcal{X}) := \left| \bigcup_{i=1}^k E(G[X_i]) \right| \leq \val(\mathcal{X})$. 
A graph $G$ is \emph{$uv$-tight} if it is $uv$-sparse and $|E|=2|V|-2$.
Note that if $G$ is
$uv$-sparse then $uv\notin E$. 
It was shown in \cite{JKN} that the edge sets of the $uv$-sparse subgraphs
of $G$ form the independent sets of a matroid, and when $|V|\geq 5$ this matroid has rank $2|V|-2$.

It is straightforward to construct $uv$-sparse graphs which are not $(2,2)$-sparse. Perhaps the simplest way is to notice that the complete bipartite graph $K_{2,3}$, with the part of size two comprising of $u$ and $v$, is clearly $(2,2)$-sparse but fails to be $uv$-sparse. To see this let $v_1,v_2,v_3$ be the vertices in the part of size three and consider the $uv$-compatible family $\mathcal{X}=\{X_1,X_2,X_3\}$ where $X_1=\{u,v,v_1\}$, $X_2=\{u,v,v_2\}$ and $X_3=\{u,v,v_3\}$. Then $i_G(\mathcal{X})=2+2+2=6$ and $\val(\mathcal{X})=(2\cdot 3-3)+(2\cdot 3-3)+(2\cdot 3-3)-2(3-1)=5$.

\section{Recursive operations}\label{sec:recursive}
 
Let $G = (V, E)$ be a graph. The \emph{$0$-extension} operation (on a pair of distinct vertices $a,b\in V$) adds a new vertex $z$ and two edges $za,zb$ to $G$. 
The \emph{1-extension} operation (on edge $ab\in E$ and vertex $c\in V\setminus \{a,b\}$)
deletes the edge $ab$, adds a new vertex $z$ and edges $za,zb,zc$.
The \emph{vertex-to-$H$ move} deletes a vertex $w$ and adds a copy of a $(2,2)$-tight graph $H$ with $V(H)\cap V=\{w\}$, along with an arbitrary replacement of each edge $xw$ by an edge of the form $xy$ with $y\in V(H)$.
A \emph{vertex-to-4-cycle move} takes a vertex $w$ with neighbours $v_1,v_2,\dots,v_k$ for any $k\geq 2$, splits $w$ into two new vertices $w,w'$ with $w'\notin V$,
adds edges $wv_1,w'v_1,wv_2,w'v_2$ and then arbitrarily replaces edges $xw$ with edges of the form $xy$ where $x\in \{v_3,\dots,v_k\}$ and $y\in \{w,w'\}$.
All $(2,2)$-tight graphs can be constructed from a single vertex by a sequence of 0- and 1-extensions, vertex-to-4-cycle and vertex-to-$K_4$ operations;
see \cite[Theorem 3.1]{nix-owe-pow-2} for more details. The operations we use are illustrated in \Cref{fig:01,fig:vk4}.

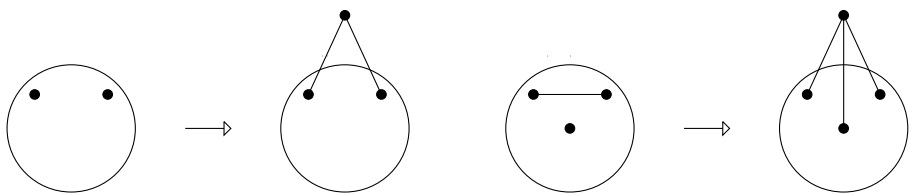
\begin{figure}[ht]
\begin{center}
\begin{tikzpicture}[scale=0.6]

\draw (-6,0) circle (40pt);

\filldraw (-6.8,.75) circle (3pt);
\filldraw (-5.2,.75) circle (3pt);

\draw[black]
(-2.65,0) -- (-3.5,0);

\draw[black]
(-2.65,0.15) -- (-2.65,-.15) -- (-2.5,0) -- (-2.65,.15);

\draw (0,0) circle (40pt);

\filldraw (-.8,.75) circle (3pt);
\filldraw (0,2.5) circle (3pt);
\filldraw (0.8,.75) circle (3pt);

\filldraw (-.5,1.6) circle (0pt) node[anchor=east]{};
\filldraw (-.4,1.6) circle (0pt) node[anchor=west]{};

\draw[black]
(.8,.75) -- (0,2.5) -- (-.8,.75);

\end{tikzpicture}
\hspace{1cm}
\begin{tikzpicture}[scale=0.6]

\draw (0,0) circle (40pt);

\filldraw (-.8,.75) circle (3pt);
\filldraw (0,0) circle (3pt);
\filldraw (0.8,.75) circle (3pt);

\filldraw (-.5,1.6) circle (0pt) node[anchor=east]{};
\filldraw (0,1.6) circle (0pt) node[anchor=west]{};

\draw
(-.8,.75) -- (0.8,.75);

\draw[black]
(2.5,0) -- (3.35,0);

\draw[black]
(3.35,0.15) -- (3.35,-.15) -- (3.5,0) -- (3.35,.15);

\draw (6,0) circle (40pt);

\filldraw (5.2,.75) circle (3pt);
\filldraw (6,2.5) circle (3pt);
\filldraw (6.8,.75) circle (3pt);
\filldraw (6,.0) circle (3pt);

\draw
(5.2,.75) -- (6,2.5) -- (6.8,.75);

\draw
(6,2.5) -- (6,0);

\end{tikzpicture}
\end{center}
\caption{0-extension and 1-extension.}
\label{fig:01}
\end{figure}

\begin{center}
\begin{figure}[ht]
\centering
\begin{tikzpicture}[scale=0.9]
\draw (0,0) circle (27pt);
\draw (5,0) circle (27pt);
 
\filldraw (0,1.5) circle (2pt);
\filldraw (.2,0.5) circle (2pt);
\filldraw (-.2,0.5) circle (2pt);
\filldraw (0.6,0.35) circle (2pt);
\filldraw (-0.6,0.35) circle (2pt);

\filldraw (4.65,1.35) circle (2pt);
\filldraw (5.35,1.35) circle (2pt);
\filldraw (4.65,1.8) circle (2pt);
\filldraw (5.35,1.8) circle (2pt);

\filldraw (5.2,0.5) circle (2pt);
\filldraw (4.8,0.5) circle (2pt);
\filldraw (5.6,0.35) circle (2pt);
\filldraw (4.4,0.35) circle (2pt);

\draw[black]
(-.6,0.35) -- (0,1.5)  -- (-.2,.5) -- (0,1.5) -- (.2,.5);
\draw[black]
(.6,.35) -- (0,1.5);

\draw[black]
(2,0) -- (2.85,0);

\draw[black]
(2.85,0.15) -- (2.85,-.15) -- (3,0) -- (2.85,.15); 

\draw[black]
(4.4,0.35) -- (4.65,1.35) -- (5.35,1.35) -- (4.65,1.8) -- (5.35,1.8) -- (5.35,1.35) -- (5.6,.35);

\draw[black]
(5.2,0.5) -- (4.65,1.8) -- (4.65,1.35) -- (4.8,.5);

\draw[black]
(4.65,1.35) -- (5.35,1.8);

\end{tikzpicture}
\hspace{1cm}
\begin{tikzpicture}[scale=0.9]
\draw (0,0) circle (27pt);
\draw (5,0) circle (27pt);

\filldraw (-.6,-.1) circle (2pt);
\filldraw (0,.4) circle (2pt);
\filldraw (.6,.1) circle (2pt);

\draw[black]
(-.8,-.1) -- (-.6,-.1) -- (0,.4) -- (.6,.1) -- (.8,.1);

\draw[black]
(-.6,-.3) -- (-.6,-.1) -- (-.6,.1);

\draw[black]
(.6,.3) -- (.6,.1) -- (.6,-.1);

\draw[black]
(-.2,.6) -- (0,.4) -- (.2,.6);

\filldraw (4.4,-.1) circle (2pt);
\filldraw (5.6,.1) circle (2pt);
\filldraw (5,.4) circle (2pt);
\filldraw (5,-.4) circle (2pt);

\draw[black]
(4.2,-.1) -- (4.4,-.1) -- (5,.4) -- (5.6,.1) -- (5,-.4) -- (4.4,-.1);

\draw[black]
(4.4,-.3) -- (4.4,-.1) -- (4.4,.1);

\draw[black]
(5.6,-.1) -- (5.6,.1) -- (5.6,.3);

\draw[black]
(5.8,.1) -- (5.6,.1);

\draw[black]
(4.8,.6) -- (5,.4);

\draw[black]
(5,.-.4) -- (5.2,-.6);

\draw[black]
(2,0) -- (2.85,0);

\draw[black]
(2.85,0.15) -- (2.85,-.15) -- (3,0) -- (2.85,.15); 

\end{tikzpicture}
\caption{The vertex-to-$H$ (with $H$ being the complete graph on 4 vertices) and vertex-to-4-cycle operations.}
\label{fig:vk4}
\end{figure}
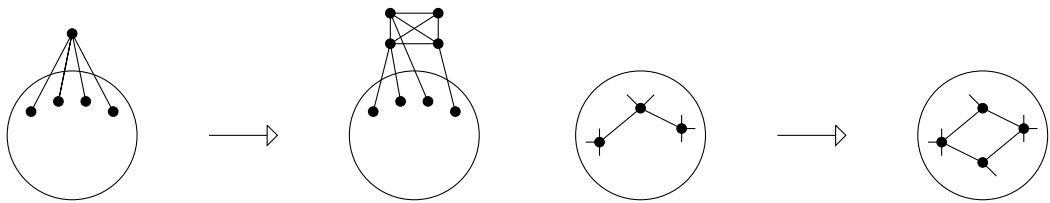
\end{center}

We shall need the following specialized versions.
First, suppose that $|V \cap \{u,v\}| = 1$.
The \emph{0-extension that adds $u$} (respectively, \emph{0-extension that adds $v$}) operation is a 0-extension where $z = u$ and $v \in V \setminus \{a,b\}$ (respectively, with $z = v$ and $u \in V \setminus \{a,b\}$).
The \emph{vertex-to-4-cycle move that adds $u$} (respectively, \emph{vertex-to-4-cycle move that adds $v$}) is a vertex-to-4-cycle move where $w=v$ and $\{w,w'\} = \{u,v\}$ (respectively, $w=u$ and $\{w,w'\} = \{u,v\}$). 
The \emph{vertex-to-$H$ move that adds $u$} (respectively, \emph{vertex-to-$H$ move that adds $v$}) is a vertex-to-$H$ move where $w = v$ and $u \in V(H) \setminus V$ (respectively, $w = u$ and $v \in V(H) \setminus V$), and the graph $H$ is $uv$-tight.

Now suppose $u,v\in V$ are two distinct vertices. 
The \emph{\(uv\)-0-extension} operation is a $0$-extension on a pair $a,b$ with \(\{a,b\}\neq \{u,v\}\).
The \emph{\(uv\)-1-extension} operation is a $1$-extension on some edge $ab$ and vertex $c$ for which $\{u,v\}$ is not a subset of $\{a,b,c\}$. 
The \emph{$uv$-vertex-to-4-cycle} and \emph{$uv$-vertex-to-$H$} moves are simply any vertex-to-4-cycle and vertex-to-$H$ moves applied to a graph containing both $u$ and $v$. 

We can immediately obtain the following result using the proof technique of \cite[Sections 5.1 and 5.2]{dew2}. 

\begin{lemma}
    Let $G$ be a graph that contains both $u$ and $v$,
    and let $G'$ be formed from $G$ by either a $uv$-0-extension or a $uv$-1-extension.
    If $G$ is $uv$-independent in a normed plane $X$,
    then $G'$ is $uv$-independent in a normed plane $X$.
\end{lemma}

The next lemma shows 0-extensions that add either $u$ or $v$ preserve independence.
It should be noted that our proof technique requires strict convexity.

\begin{lemma}
\label{l:zeroext}
    Let $G=(V,E)$ be a graph that contains $u$ but not $v$, and let $X$ be a strictly convex normed plane.
    Suppose $G'$ is formed from $G$ by a 0-extension that adds $v$.
    Then $G'$ is $uv$-independent if and only if $G$ is independent.
\end{lemma}

\begin{proof}
    We note that as $G'$ contains $G$ as a subgraph,
    if $G'$ is $uv$-independent then $G$ will be independent.
    Suppose there is an independent placement $p$ of $G$ in $X$.
    By applying translations, we may suppose that $p_u = 0$.
    Let $v_1,v_2$ be the two neighbours of $v$ in $G'$.
    We may also assume that $p_{v_1}$ and $p_{v_2}$ are linearly independent and smooth;
    indeed if this was not true, we could apply \Cref{lem:perturb} to $(G,p)$ to find a placement of $G$ where it is true.
    Define $p'$ to be the well-positioned placement of $G'$ with $p'_x = p_x$ for all $x \in V$ and $p'_v = p_u$.
    From our choice of placement of $G'$, we see that
    \begin{align*}
        R(G',p') =
        \left[
        \begin{array}{c|c}
            R(G,p) & \mathbf{0}_{|E| \times 2} \\
            \hline
            A & -\varphi_{p_{v_1}} \\
            B & -\varphi_{p_{v_2}}
        \end{array}
        \right]
    \end{align*}
    for some $1 \times 2|V|$ matrices $A$ and $B$.
    Hence $(G',p')$ is independent if and only if $\varphi_{p_{v_1}}, \varphi_{p_{v_2}}$ are linearly independent.
    Since $p_{v_1},p_{v_2}$ are linearly independent and $X$ is strictly convex,
    the pair $\varphi_{p_{v_1}}, \varphi_{p_{v_2}}$ are linearly independent as required.
\end{proof}

For the vertex-to-4-cycle move we will use the technique of \cite[Lemma 11]{JKN} to show that a vertex-to-4-cycle move which creates two coincident vertices preserves independence.
Similarly to the previous result, we will require that the normed plane in question is strictly convex.

\begin{lemma}
\label{l:fourcycleind}
    Let $G=(V,E)$ and $G'=(V',E')$ be graphs and let $X$ be a strictly convex normed plane.
    \begin{enumerate}[(i)]
        \item If $G$ is independent in $X$ and $G'$ is formed from $G$ by a vertex-to-4-cycle move that adds either $u$ or $v$, then $G'$ is $uv$-independent in $X$.
        \item If $G$ is $uv$-independent in $X$ and $G'$ is formed from $G$ by a $uv$-vertex-to-4-cycle move, then $G'$ is $uv$-independent in $X$.
    \end{enumerate}
\end{lemma}

\begin{proof}
Suppose that $G$ is $uv$-independent (respectively, independent).
Using \Cref{lem:uv-perturb} (respectively, \Cref{lem:perturb}),
choose a $uv$-independent (respectively, independent) placement $p$ of $G$ in $X$ so that $p_w, p_{v_1}, p_{v_2}$ are not collinear.
By applying translations to $p$, we shall assume that $p_w = 0$.
Now define $p'$ to be the placement of $G'$ with $p'_x=p_x$ for all $x \in V$ and $p'_{w'} = p_w$.
The pair $(G',p')$ form a well-positioned $uv$-coincident framework due to our choice of $p'$.
Since $X$ is strictly convex,
the pair $\varphi_{p_{v_1}},\varphi_{p_{v_2}}$ are linearly independent.
Define $G''$ to be the graph formed from $G'$ by replacing each edge $w'v_i$ for $3 \leq i \leq k$ with the edge $w v_i$.
Then
\begin{align*}
    R(G'',p') = 
    \left[
    \begin{array}{c|c}
    R(G, p) &  \mathbf{0}_{|E| \times 2} \\
    \hline
    A & \varphi_{p'_{w'} - p'_{v_1}} \\
    B & \varphi_{p'_{w'} - p'_{v_2}}
    \end{array}
    \right] = 
    \left[
    \begin{array}{c|c}
    R(G, p) &  \mathbf{0}_{|E| \times 2} \\
    \hline
    A & -\varphi_{p_{v_1}} \\
    B & -\varphi_{p_{v_2}}
    \end{array}
    \right],
\end{align*}
for some $1 \times 2|V|$ matrices $A$ and $B$.
Since $p_{v_1},p_{v_2}$ are linearly independent and $X$ is strictly convex,
the pair $\varphi_{p_{v_1}}, \varphi_{p_{v_2}}$ are linearly independent.
Hence $R(G'',p')$ has linearly independent rows.
To prove that $G'$ is $uv$-independent in $X$ we will describe a series of rank-preserving row operations that will form $R(G',p')$ from $R(G'',p')$.

As $\varphi_{p_{v_1}}$ and $\varphi_{p_{v_2}}$ are linearly independent,
there exist for each $3 \leq i \leq k$ a unique pair of values $\alpha_i$ and $\beta_i$ such that
\begin{align*}
    \alpha_i \varphi_{p_{v_1}} + \beta_i \varphi_{p_{v_2}} = \varphi_{p_{v_i}} = \varphi_{p'_{v_i} - p'_{z}},
\end{align*}
where $z \in \{w,w'\}$ is chosen so that $v_i z \in E(G')$.
For $1 \leq i \leq k$, 
let $(wv_i)$ denote the row of $R(G'',p')$ corresponding to the edge $wv_i$, 
and similarly let $(w'v_1)$ and $(w'v_2)$ denote the rows of $R(G'',p')$ corresponding to edges $w'v_1$ and $w'v_2$ respectively.
For $v_i \in N_{G'}(w')$,
let $[w'v_i]$ denote the row of $R(G', p')$ corresponding to the edge $w'v_i$.
Now, for all $v_i \in N_{G'}(w') \backslash \{v_1, v_2\}$, 
we have 
\begin{align*}
    [w'v_i] = (wv_i)-\alpha_i(wv_1)-\beta_i(wv_2)+\alpha_i(w'v_1)+\beta_i(w'v_2).
\end{align*}
These row operations, when applied $R(G'',p')$, preserve linear independence and form the matrix $R(G', p')$. 
Therefore the rows of $R(G', p')$ are linearly independent.
\end{proof}

We now prove that vertex-to-$H$ operations that add either $u$ or $v$ and $uv$-vertex-to-$H$ operations will preserve $uv$-independence.

\begin{lemma}
\label{lem:extension}
    Let $G=(V,E)$ and $G'=(V',E')$ be graphs and let $X$ be any normed plane.
    \begin{enumerate}[(i)]
        \item\label{lem:extension1} Suppose $G$ is independent in $X$ and $G'$ is formed from $G$ by a vertex-to-$H$ move that adds either $u$ or $v$.
        If $H$ is minimally $uv$-rigid in $X$, then $G'$ is $uv$-independent in $X$.
        \item\label{lem:extension2} Suppose $G$ is $uv$-independent in $X$ and $G'$ is formed from $G$ by a $uv$-vertex-to-$H$ move.
        If $H$ is minimally rigid in $X$, then $G'$ is $uv$-independent in $X$.
    \end{enumerate}
\end{lemma}

\begin{proof}
If (\ref{lem:extension1}) holds, let $(H,q)$ be a minimally rigid $uv$-coincident framework in $X$ and $(G,p)$ be an independent framework in $X$,
    while if (\ref{lem:extension2}) holds, let $(H,q)$ be a minimally rigid framework in $X$ and $(G,p)$ be an independent $uv$-coincident framework in $X$.
    By applying translations we may assume $q_w = p_w = 0$.
    For any matrix $A$ with columns corresponding to a vertex subset of $V(G) \cup V(H)$,
    define $A_w$ to be the matrix where we delete all columns corresponding to the vertex $w$.
    Given the fixed basis $b_1,b_2 \in X$ used to define our rigidity matrices in $X$,
    we define the matrix
    \begin{align*}
        M :=
        \left[
        \begin{array}{c|c}
            R(H,q)_w & \mathbf{0}_{|E(H)| \times (2|V|-2)} \\
            \hline
            A & R(G,p)_w
        \end{array}
        \right]
    \end{align*}
    where $A$ is the $|E| \times (2|V(H)|-2)$ matrix with entries
    \begin{align*}
        A_{e,(y,i)} =
        \begin{cases}
            \varphi_{p_y-p_w}(b_i) &\text{if } e = xw,\\
            0 &\text{otherwise.}
        \end{cases}
    \end{align*}
    By our choices of $p$ and $q$,
    the matrix $M$ has linearly independent rows.
    
    For each $n \in \mathbb{N}$,
    choose a well-positioned $uv$-coincident framework $(G',p^n)$ where $p^n_x = q_x / n$ for each $x \in V(H)$ and $\|p^n_x - p_x\|< 1/n$ for each $x \in V$ (this framework can be seen to exist from \Cref{lem:uv-perturb}).
    Define $M_n$ to be the matrix formed from multiplying each row of $R(G',p^n)_w$ corresponding to an edge of $H$ by $n$.
    As the map $x \mapsto \varphi_x$ is continuous on the set of smooth points of $X$ (see \cite[Theorem 25.5]{rockafellar}),
    the sequence of matrices $(M_n)_{n \in \mathbb{N}}$ will converge to $M$.
    Hence for sufficiently large $N \in \mathbb{N}$,
    the matrix $M_{n_0}$ (and hence $R(G',p^{n_0})_w$) will have linearly independent rows.
    By setting $p'=p^{n_0}$, we obtain our desired independent $uv$-coincident framework $(G',p')$.
\end{proof}

\section{Characterising coincident point independence}
\label{sec:char}

With the geometric results of the previous section in hand,
we can use the combinatorics of \cite{JKN} to prove the difficult sufficiency direction of our main result on coincident frameworks.
We begin with the following result which can be extracted from the proof of \cite[Theorem 4]{JKN}.

\begin{proposition}[\cite{JKN}]\label{p:uv-construction}
    Any $uv$-tight graph on at least five vertices can be constructed from either a $(2,2)$-tight graph with at least four vertices that contains exactly one of $u$ and $v$, or from the graph consisting of two copies of $K_4$ intersecting in a single vertex $x\notin \{ u,v\}$ where $u$ and $v$ are in different copies of $K_4$ (see \Cref{fig:2k4s}), by a sequence of 0-extensions that add $u$ or $v$, vertex-to-4-cycle and vertex-to-$H$ moves that add $u$ or $v$,
    $uv$-0- and $uv$-1-extensions, and
    $uv$-vertex-to-4-cycle and $uv$-vertex-to-$H$ moves.
\end{proposition}

 \begin{figure}[htp]
     \begin{tikzpicture}[scale=0.8]
    \node (v) at (2.4,0) {$v$};
    \node (u) at (-2.4,0) {$u$};
    \node (x) at (0,-0.4) {$x$};
    
    \node[vertex] (a) at (2,0) {};
	\node[vertex] (b) at (1,-1) {};
    \node[vertex] (c) at (1,1) {};
    
	\node[vertex] (d) at (0,0) {};
	
    \node[vertex] (e) at (-2,0) {};
	\node[vertex] (f) at (-1,-1) {};
    \node[vertex] (g) at (-1,1) {};
	
	\draw[edge] (a)edge(b);
	\draw[edge] (a)edge(c);
	\draw[edge] (a)edge(d);
	\draw[edge] (b)edge(c);
	\draw[edge] (b)edge(d);
	\draw[edge] (c)edge(d);
    
	\draw[edge] (d)edge(e);
	\draw[edge] (d)edge(f);
	\draw[edge] (d)edge(g);
	\draw[edge] (e)edge(f);
	\draw[edge] (e)edge(g);
	\draw[edge] (f)edge(g);
\end{tikzpicture}
\caption{A $uv$-tight graph that is one of the base graphs of the construction described in \Cref{p:uv-construction}.}
\label{fig:2k4s}
\end{figure}
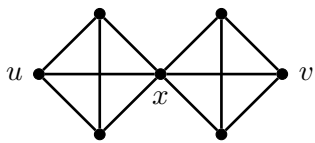

We will also require the following lemmas.

\begin{lemma}\label{l:small}
    Let $G=(V,E)$ be a graph with at most 4 vertices that contains both $u$ and $v$,
    and let $X$ be a normed plane.
    Then $G$ is $uv$-sparse if and only if it is $uv$-independent in $X$.
\end{lemma}

\begin{proof}
    The only graphs on 4 or fewer vertices that are not $uv$-sparse are those which contain the edge $uv$, and if $G$ contains the edge $uv$ then it is not $uv$-independent.
    Suppose $uv \notin E$.
    We note that $G$ must be a subgraph of $K_4-uv$,
    so it is sufficient to consider the case $G=K_4-uv$.
    As $G$ can be formed from $G-u$ by a 0-extension that adds $u$,
    $G$ is $uv$-independent by \Cref{thm:rigidne} and \Cref{l:zeroext}.
\end{proof}

\begin{lemma}\label{l:base}
    Let $G=(V,E)$ be the graph consisting of two copies of $K_4$ intersecting in a single vertex $x\notin \{ u,v\}$,
    where $u$ and $v$ are in different copies of $K_4$.
    Then $G$ is minimally $uv$-rigid in any normed plane $X$.
\end{lemma}

\begin{proof}
    Let $V_u = \{x,u,a_u,b_u\}$ and $V_v = \{x,v,a_v,b_v\}$ be the two distinct cliques of size 4 in $G$.
    By \Cref{thm:rigidne},
    there exists a placement $p^u : V_u \rightarrow X$ so that the framework $(K_{V_u},p^u)$, where $K_{V_u}$ is the complete graph with vertex set $V_u$, is minimally rigid in $X$.
    Define the placement $p : V \rightarrow X$ by setting $p_{a_v} = p^u_{a_u}$, $p_{b_v} = p^u_{b_u}$, $p_v = p^u_u$, and $p_y = p^u_y$ for all $y \in V_u$.
    We now note that $(G,p)$ is a minimally rigid $uv$-coincident framework;
    this follows from the fact that joining two minimally rigid frameworks in a normed plane produces a minimally rigid framework, since the trivial infinitesimal flexes correspond only to translations.
    Hence $G$ is minimally $uv$-rigid as required.
\end{proof}

\begin{thm}
\label{thm:uvmatroid}
    A graph is $uv$-independent in a strictly convex normed plane $X$ if and only if it is $uv$-sparse.
\end{thm}

\begin{proof}
    First suppose $G$ is $uv$-independent in $X$.
    Let $G/uv$ denote the graph obtained from \(G\) by contracting the vertex pair $u,v$ into a new vertex which we denote as $z$\footnote{For us, a contraction will always be the more general vertex-contraction (which does not require $u$ and $v$ be adjacent) not the stricter edge-contraction (which does require $u$ and $v$ be adjacent).}.
    Let $(G,p)$ be a regular (and hence independent) $uv$-coincident framework in $X$.
    We obtain a framework $(G/uv,p^{uv})$ in $X$ by putting $p^{uv}_{z}=p_u = p_v$ and $p^{uv}_x = p_x$ for all $x\in V\setminus \{u,v\}$.
    For any $U \subseteq V$,
    the (possibly $uv$-coincident) induced subframework $(G[U],p|_U)$ is independent.
    Hence,
    if $\{u,v\}\not\subseteq U$,
    then $i_G(U)\leq \val(U)$ by \Cref{thm:rigidne}.
    Since the case when $U = \{u, v\}$ is trivial, it now remains to show that $i_G(\mathcal{X})\leq \val (\mathcal{X})$ for all $uv$-compatible families $\mathcal{X}$ in $G$. 
    (Note that the case when $U\subseteq V$ and $\{u,v\}\subseteq U$ will be included by taking $\mathcal{X}=\{U\}$.)

    Let \(\mathcal{X}=\{X_{1},\dots,X_{k}\}\) be a \(uv\)-compatible
    family and consider the subgraph
    \(H=(U,F)\) of $G$,
    where $U=\bigcup_{i=1}^k X_i$ and $F=\bigcup_{i=1}^k E(G[X_i])$.
    By contracting the vertex pair $u,v$ in $H$, we obtain the graph $H/uv$.
    Define $q$ to be the restriction of $p$ to the vertex set $U$ and $q^{uv}$ to be the restriction of $p^{uv}$ to the vertex set $U - \{u,v\} + z$.
    We have
    \(\mathcal{X}_{uv}=\{X_{1}/uv,\dots,X_{k}/uv\}\) is a cover of $E(H/uv)$ where
    \(X_{i}/uv\) denotes the set that we get from \(X_{i}\) by
    identifying \(u\) and \(v\).
    By \Cref{c:rankformula},
     we have
    \begin{eqnarray*}
        \rank R(H/uv,q^{uv}) &\leq& \sum_{i=1}^k(2|X_i/uv|-2-s(|X_i/uv|)\\
        &=& \sum_{i=1}^k(2|X_i|-2-t_{X_i}) \\
        &=& \val(\mathcal{X}) - 2.
    \end{eqnarray*}
    Every vector $\mu^{uv}$ in the kernel of $R(H/uv,q^{uv})$ determines a unique vector $\mu$ in the kernel of $R(H,q)$ with $\mu_u=\mu_v =\mu^{uv}_{z}$ and $\mu_x = \mu_x^{uv}$ for all for all $x\in U\setminus \{u,v\}$.
    Hence $\dim \ker R(H,q) \geq \dim\ker R(H/uv,q^{uv})$.
    The rigidity matrix $R(H,q)$ has linearly independent rows since $R(G,p)$ has linearly independent rows,
    hence we have
    \begin{align*}
        i_G(\mathcal{X}) = \rank R(H,q) \leq \rank R(H/uv,q^{uv})+ 2 \leq \val(\mathcal{X}).
    \end{align*}
    Thus \(G\) is $uv$-sparse.
    
    We prove the sufficiency by induction on $|V|$. Suppose that $G$ is $uv$-sparse.
    If $|V|\leq 4$, then $G$ is $uv$-independent in $X$ by \Cref{l:small}. 
    So we may suppose that $|V|\geq 5$.
    By adding additional edges, if necessary,  we may assume $G$ is $uv$-tight\footnote{Recall that $uv$-sparse graphs are the independent sets of a matroid, and when $|V| \geq 5$, the bases of this matroid have rank $2|V|-2$.}. 
    By \Cref{p:uv-construction}, $G$ can be constructed from either a $(2,2)$-tight graph containing exactly one of $u$ and $v$, or the graph pictured in \Cref{fig:2k4s}, by the operations defined in \Cref{sec:recursive}.
    Furthermore, as $X$ is strictly convex, the corresponding geometric operations preserve minimal rigidity in $X$ (see \Cref{sec:recursive}).
    The result now follows from \Cref{thm:rigidne} (i.e., every $(2,2)$-tight graph is independent in $X$) and \Cref{l:base}.
\end{proof}

We next use this result to prove the following delete-contract characterisation of $uv$-rigidity.

\begin{thm}\label{lem:dc}
Let $G$ be a graph with distinct vertices $u,v$, and let $X$ be a strictly convex normed plane. 
Then $G$ is $uv$-rigid in $X$ if and only if $G-uv$ and $G/uv$ are both rigid in $X$.
\end{thm}

\begin{proof}
Suppose that $G$ is $uv$-rigid.
It is immediate from the definition that $G-uv$ must be rigid.
Choose a regular $uv$-coincident placement $p$ of $G$,
and define $p^{uv}$ to be the placement of $G/uv$ where $p^{uv}_x = p_x$ for all $x \in V - \{u,v\}$ and (given that $z$ is the vertex obtained from $u$ and $v$ during the contraction) $p^{uv}_z = p_u = p_v$.
Given an infinitesimal flex $\mu^{uv}$ of $(G/uv,p^{uv})$,
we can form an infinitesimal flex $\mu$ of $(G,p)$ by setting $\mu_x = \mu^{uv}_x$ for all $x \in V - \{u,v\}$ and $\mu_u = \mu_v = \mu^{uv}_z$.
Since $(G,p)$ is infinitesimally rigid as a $uv$-coincident framework,
we must have that $\mu = (\lambda)_{x\in V}$ (and hence $\mu^{uv} = (\lambda)_{x\in V(G/uv)}$) for some vector $\lambda \in X$.
Thus $(G/uv,p^{uv})$ is infinitesimally rigid and $G/uv$ is rigid.

The converse follows from \Cref{thm:uvmatroid} as in the proof of \cite[Theorem 1]{JKN}.
\end{proof}

We conjecture that the last two results
apply in arbitrary normed planes.

\begin{conjecture}
Let \(G=(V,E)\) be a graph and let $u,v\in V$ be distinct vertices. Then \(G\) is $uv$-independent in a normed plane $X$ if and only if \(G\) is $uv$-sparse.
\end{conjecture}

Indeed extending our proof to this generality requires only improvements to \Cref{l:zeroext,l:fourcycleind}.

\section{Global rigidity in analytic normed planes}
\label{sec:glob}

A framework $(G,p)$ in a normed space $X$ is said to be \emph{globally rigid} if every other framework $(G,q)$ in $X$ with $\|p_v-p_w\| = \|q_v-q_w\|$ for every edge $vw \in E$ is congruent to $(G,p)$.
A graph is then said to be \emph{globally rigid} in $X$ if the set
\begin{align*}
    \left\{ p \in X^V : (G,p) \text{ is globally rigid}  \right\}
\end{align*}
has a non-empty interior.
It can be quickly seen that any globally rigid framework/graph will also be rigid.

Although much is known about global rigidity in Euclidean spaces, very little is known about the property for normed spaces.
The results that are known are only for \emph{analytic normed spaces}, i.e., normed spaces where the norm restricted to the non-zero points is a real analytic function.
As well as being strictly convex (\cite[Lemma 3.1]{DN}), analytic normed spaces have many useful properties, including the following.

\begin{lemma}\label{l:openconull}
    Let $G$ be a graph with distinct vertices $u,v$ and let $X$ be an analytic normed space.
    \begin{enumerate}[(i)]
        \item\label{l:openconull1} The set of all $p \in X^V$ where $(G,p)$ is a regular framework is an open conull subset of $X^V$.
        \item\label{l:openconull2} The set of all $p \in X^V/uv$ where $(G,p)$ is a regular $uv$-coincident framework is an open conull subset of $X^V/uv$.
    \end{enumerate}
\end{lemma}

\begin{proof}
    If $\dim X = 1$ then the result follows immediately from noticing that all well-positioned frameworks and $uv$-coincident frameworks are regular.
    Suppose $\dim X \geq 2$.
    It was shown in \cite[Proposition 3.2]{DN} that the set of well-positioned but non-regular placements of $G$ are exactly the zero set of a non-constant analytic function defined on the connected open conull set of well-positioned placements. This gives (i).
    For (ii) we can use the same technique to show that the set of well-positioned but non-regular $uv$-coincident placements of $G$ are exactly the zero set of a non-constant analytic function defined on the connected open conull set of well-positioned $uv$-coincident placements.
    The result now holds as the zero set of a non-constant analytic function with connected domain is always a closed null subset (see \cite[Proposition 2.3]{DN}).
\end{proof}

Importantly, we can define a large class of globally rigid graphs in any analytic normed plane.

\begin{proposition}[\cite{DN}]\label{thm:k5-eH} 
Let $X$ be an analytic normed plane.
	Then the graphs $K_5-e$ and $H$, depicted in \Cref{fig:smallgraphs1}, are globally rigid in $X$. 
	Moreover any graph obtained from either of these by a sequence of degree 3 vertex additions (i.e., add a vertex and join it to three other vertices) and edge additions is globally rigid.
\end{proposition}

 \begin{figure}[htp]
\begin{center}
\begin{tikzpicture}[scale=.4]
\filldraw (-.5,0) circle (4.5pt)node[anchor=east]{};
\filldraw (0,3) circle (4.5pt)node[anchor=east]{};
\filldraw (3.5,0) circle (4.5pt)node[anchor=west]{};
\filldraw (3,3) circle (4.5pt)node[anchor=west]{};
\filldraw (1.5,-1.5) circle (4.5pt)node[anchor=west]{};

 \draw[black,thick]
(1.5,-1.5) -- (-.5,0) -- (0,3) -- (3.5,0) -- (3,3) -- (-.5,0) -- (3.5,0);

\draw[black,thick]
(1.5,-1.5) -- (0,3) -- (3,3) -- (1.5,-1.5);

        \end{tikzpicture}
          \hspace{0.5cm}
     \begin{tikzpicture}[scale=.4]
\filldraw (0,0) circle (4.5pt)node[anchor=east]{};
\filldraw (0,3.5) circle (4.5pt)node[anchor=east]{};
\filldraw (3.5,0) circle (4.5pt)node[anchor=north]{};
\filldraw (3.5,3.5) circle (4.5pt)node[anchor=south]{};
\filldraw (7,0) circle (4.5pt)node[anchor=west]{};
\filldraw (7,3.5) circle (4.5pt)node[anchor=west]{};

 \draw[black,thick]
(0,0) -- (0,3.5) -- (3.5,0) -- (3.5,3.5) -- (0,0) -- (3.5,0) -- (7,3.5);

\draw[black,thick]
(0,3.5) -- (3.5,3.5) -- (7,3.5) -- (7,0) -- (3.5,3.5);

\draw[black,thick]
(7,0) -- (3.5,0);

\end{tikzpicture}
\end{center}
\vspace{-0.3cm}
\caption{The graphs $K_5^-$ (left) and $H$ (right).}
\label{fig:smallgraphs1}
\end{figure}
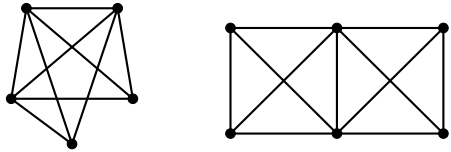

We next increase this class of graphs with the following construction operation introduced in \cite{JNglobal}. A \emph{generalised
vertex split}, is defined as follows.
Choose $z\in V$ and a
partition $N_u,N_v$ of the neighbours of $z$.
Next, delete $z$ from
$G$ and add two new vertices $u,v$ joined to $N_u,N_v$,
respectively.
Finally add two new edges $uv,uw$ for some $w\in
V\setminus N_u$.
See \Cref{fig:vsplit} for an illustration of the operation.

\begin{center}
\begin{figure}[ht]
\begin{tikzpicture}[scale=0.7]
\draw (-3,-5) circle (35pt);
\draw (4,-5) circle (35pt);

\filldraw (-3,-2.5) circle (2pt) node[anchor=south]{$z$};
\filldraw (3.5,-2.5) circle (2pt) node[anchor=south]{$u$};
\filldraw (4.5,-2.5) circle (2pt) node[anchor=south]{$v$};

\filldraw (-3,-5) circle (2pt) node[anchor=north]{$w$};
\filldraw (4,-5) circle (2pt) node[anchor=north]{$w$};

\filldraw (-4,-4.5) circle (2pt);
\filldraw (-3.5,-4.5) circle (2pt);
\filldraw (-2.5,-4.5) circle (2pt);
\filldraw (-2,-4.5) circle (2pt);
\filldraw (3,-4.5) circle (2pt);
\filldraw (3.5,-4.5) circle (2pt);
\filldraw (4.5,-4.5) circle (2pt);
\filldraw (5,-4.5) circle (2pt);

\draw[black]
(0,-5) -- (1,-5) -- (0.9,-5.1);

\draw[black]
(1,-5) -- (0.9,-4.9);

\draw[black]
(-4,-4.5) -- (-3,-2.5) -- (-3.5,-4.5);

\draw[black]
(-2.5,-4.5) -- (-3,-2.5) -- (-2,-4.5);

\draw[black]
(3,-4.5) -- (3.5,-2.5) -- (3.5,-4.5);

\draw[black]
(4.5,-4.5) -- (4.5,-2.5) -- (5,-4.5);

\draw[black]
(4,-5) -- (3.5,-2.5) -- (4.5,-2.5);

\end{tikzpicture}
\caption{Generalised vertex split.} \label{fig:vsplit}
\end{figure}
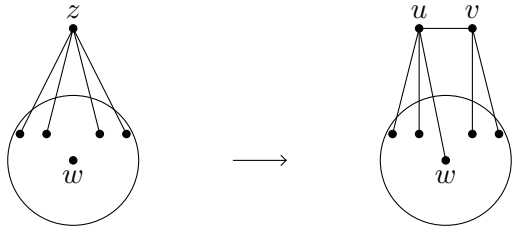
\end{center}

As the name suggests, this operation generalises the usual vertex splitting operation, see \cite{Whi}, which is the special case when $w$ is chosen to be a neighbour of $v$. Note also that the special case when $u$ has degree 3 (and $v=z$) is the well known 1-extension operation. Previously it was not known whether the 1-extension operation or a suitably restricted version of the vertex splitting operation preserves global rigidity in any non-Euclidean normed plane $X$. 

As an application of our main result we will deduce that global rigidity can, under certain conditions, be preserved for generalised vertex splits.
We will first need the following result which can be seen to follow from adapting the methods in \cite[Section 3.2]{DN} to allow frameworks with zero-length edges\footnote{Although it is a prerequisite in \cite[Section 3.2]{DN} that the frameworks are well-positioned, the proof technique only requires that the squared edge-length map is differentiable. Since the map $x \mapsto \|x\|^2$ is always differentiable at the point 0, we can refine the result so that it holds for frameworks with zero-length edges.}.

\begin{lemma}\label{l:peturbgr}
    Let $(G,p)$ be a $uv$-coincident framework in a smooth normed space $X$.
    If $(G,p)$ is globally rigid and infinitesimally rigid,
    then there exists an open neighbourhood $U \subset X^V$ of $p$ where for each $q \in U$ the framework $(G,q)$ is globally rigid.
\end{lemma}

\begin{thm}\label{thm:vsplitglobal}
    Let $G$ be a globally rigid graph in an analytic normed plane $X$.
    Let $G'$ be a generalised vertex split of $G$ at the vertex $z$ with new vertices $u,v$ and suppose that $G'-uv$ is rigid in $X$.
    Then $G'$ is globally rigid in $X$.
\end{thm}

\begin{proof} 
    Since $G'/uv = G$ is globally rigid in $X$ it is also rigid in $X$ by \Cref{t:ar}.
    As $G'-uv$ is also rigid in $X$,
    \Cref{lem:dc} implies that $G'$ is $uv$-rigid in $X$.
    Hence by \Cref{l:openconull},
    we may choose an infinitesimally and globally rigid framework $(G,p)$ so that if we define $(G',p')$ to be the $uv$-coincident framework with $p'_x = p_x$ for all $x \in V$ and $p'_u = p'_v = p_z$,
    then $(G', p')$ will be infinitesimally rigid also.
    Furthermore, $(G',p')$ will also be globally rigid as $(G,p)$ is globally rigid.
    We can now use \Cref{l:peturbgr} to deduce that $(G',q)$ is globally rigid in
    $X$ for all $q$ sufficiently close to $p'$.
    Hence $G'$ is globally rigid in $X$ also.
\end{proof}

We can now improve upon \Cref{thm:k5-eH}.
Here a graph $G = (V,E)$ is \emph{redundantly rigid} in $X$ if $G-e$ is rigid in $X$ for any edge $e \in E$.

\begin{corollary}\label{cor:improvement}
Let $G$ be a graph obtained from $K_5^-$ or $H$ by a sequence of generalised vertex splits that preserve redundant rigidity, edge additions and degree at least 3 vertex additions. Then $G$ is globally rigid in any analytic normed plane.
\end{corollary}

\begin{proof}
Follows immediately from \Cref{thm:k5-eH} and \Cref{thm:vsplitglobal}.
\end{proof}

Since minimally rigid graphs in $X$ have $2|V|-2$ edges by \Cref{thm:rigidne}, it is natural to expect that
if $G=(V,E)$ is globally rigid then $|E|\geq 2|V|-1$. 
The graphs $K_5^-$ and $H$ both achieve equality, but the inequality is strict for every graph in the infinite family obtained from these as in \Cref{thm:k5-eH}. To illustrate the power of \Cref{cor:improvement} we note that
we now have infinitely many globally rigid graphs for which equality holds and that this still holds if we restrict generalised vertex splitting to just one of vertex splitting or 1-extension. Two examples are depicted in \Cref{fig:examples}. The graph on the left is obtained from $H$ by a vertex split and the graph on the right is obtained from $H$ by a 1-extension. Both are globally rigid in $X$ by \Cref{cor:improvement}.

\begin{center}
\begin{figure}[ht]
\begin{tikzpicture}[scale=0.7]

\filldraw (0,0) circle (2pt) node[anchor=south]{};
\filldraw (2,0) circle (2pt) node[anchor=south]{};
\filldraw (0,2) circle (2pt) node[anchor=south]{};
\filldraw (2,2) circle (2pt) node[anchor=south]{};
\filldraw (4,0) circle (2pt) node[anchor=south]{};
\filldraw (5.5,1) circle (2pt) node[anchor=south]{};
\filldraw (4.5,2.5) circle (2pt) node[anchor=south]{};

\draw[black]
(0,0) -- (0,2) -- (2,0) -- (2,2) -- (0,0) -- (2,0) -- (4,0) -- (5.5,1) -- (4.5,2.5) -- (4,0) -- (2,2) -- (5.5,1);

\draw[black]
(0,2) -- (2,2) -- (4.5,2.5);

\filldraw (13,0) circle (2pt) node[anchor=south]{};
\filldraw (9,0) circle (2pt) node[anchor=south]{};
\filldraw (11,0) circle (2pt) node[anchor=south]{};
\filldraw (13,2) circle (2pt) node[anchor=south]{};
\filldraw (9,2) circle (2pt) node[anchor=south]{};
\filldraw (11,2) circle (2pt) node[anchor=south]{};
\filldraw (10,3) circle (2pt) node[anchor=south]{};

\draw[black]
(9,0) -- (11,0) -- (13,0) -- (13,2) -- (11,2) -- (9,2) -- (9,0) -- (11,2) -- (13,0);

\draw[black]
(11,0) -- (13,2) -- (10,3) -- (11,0) -- (11,2);

\draw[black]
(9,2) -- (10,3);

\end{tikzpicture}
\caption{Examples of globally rigid graphs.} \label{fig:examples}
\end{figure}
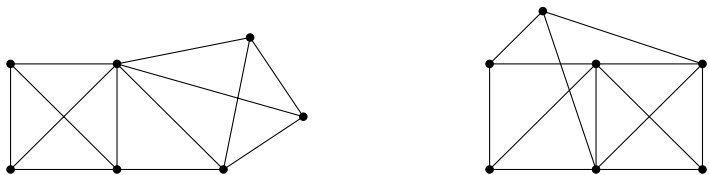
\end{center}

\bibliographystyle{abbrv}
\def\lfhook#1{\setbox0=\hbox{#1}{\ooalign{\hidewidth
  \lower1.5ex\hbox{'}\hidewidth\crcr\unhbox0}}}

\end{document}